\documentclass[12pt]{amsart}

\usepackage{amssymb}

\usepackage{enumitem}

\usepackage{graphicx}

\makeatletter
\@namedef{subjclassname@2020}{%
  \textup{2020} Mathematics Subject Classification}
\makeatother

\usepackage[T1]{fontenc}

\usepackage[T1]{fontenc}

\usepackage{setspace}

\usepackage{hyperref}
\usepackage{cleveref}
\usepackage{xfrac}

\newcommand{\mT}{\mathbb{T}}
\newcommand{\T}{\overline{T}}
\newcommand{\spt}{{\rm Sub}^{p}_{\mathbb{T}^{n}}}
\let\ol=\overline
\let\mi=\setminus
\newcommand{\Aut}{{\rm Aut}}

\newcommand{\Sub}{{\rm Sub}}

\newcommand{\N}{{\mathbb N}}
\newcommand{\Z}{{\mathbb Z}}
\newcommand{\R}{{\mathbb R}}
\newcommand{\C}{{\mathbb C}}

\newcommand{\Sc}{{\mathbb S}}
\newcommand{\mfS}{{\mathfrak{S}}}
\newcommand{\mfH}{{\mathfrak{H}}}
\newcommand{\gf}{{\mathfrak g}}
\newcommand{\hf}{{\mathfrak h}}

\newcommand{\A}{{\mathcal A}}
\newcommand{\Hc}{{\mathcal H}}

\newcommand{\K}{{\mathcal K}}
\newcommand{\D}{{\mathcal D}}
\newcommand{\U}{{\mathcal U}}

\newcommand{\RP}{{\mathbb{RP}}}

\newcommand{\spg}{{\rm Sub}^p_G}
\newcommand{\spp}{{\rm Sub}^{p}}

\newcommand{\Id}{\,{\rm Id}}

\newcommand{\GL}{{\rm GL}}
\newcommand{\SO}{{\rm SO}}

\newcommand{\PSL}{{\rm PSL}}

\newtheorem{theorem}{Theorem}[section]
\newtheorem{corollary}[theorem]{Corollary}
\newtheorem{lemma}[theorem]{Lemma}
\newtheorem{proposition}[theorem]{Proposition}

\theoremstyle{definition}

\newtheorem{remark}[theorem]{Remark}

\numberwithin{equation}{section}

\theoremstyle{remark}

\begin{document}

\title[Automorphism Action on the space of one-parameter subgroups]{Dynamics of actions of automorphisms on the space of one-parameter subgroups of a torus and applications}

\author[D. Chatterjee]{Debamita Chatterjee}
\address{School of Physical Sciences, Jawaharlal Nehru University, New Delhi 110067, India 
(Current Address: Stat-Math Unit, Indian Statistical Institute (ISI), 8th Mile, Mysore Road, Bangalore 560059, India) }
\email{debamita.math@gmail.com}

\author[H. Lekharu]{Himanshu Lekharu}
\address{School of Physical Sciences, Jawaharlal Nehru University, New Delhi 110067, India}
\email{himanifolds@gmail.com}

\author[R. Shah]{Riddhi Shah}
\address{School of Physical Sciences, Jawaharlal Nehru University, New Delhi 110067, India}
\email{rshah@jnu.ac.in, riddhi.kausti@gmil.com}

\date{}

\subjclass[2020]{Primary 37B05; Secondary 22E15, 22D45}

\keywords{distal and expansive actions of automorphisms, 
 Lie groups, \\ 
 $n$-torus, one-parameter subgroups, Chabauty topology.}
 
\begin{abstract}
For a connected Lie group $G$, we study the dynamics of actions of automorphisms of $G$ on certain compact invariant subspaces 
of closed subgroups of $G$ in terms of distality and expansivity. We show that only the finite order automorphisms of $G$ act distally 
on Sub$^p_G$, the smallest compact space containing all closed one-parameter subgroups of $G$, when $G$ is any $n$-torus, $n\in\N$. 
This enables us to relate distality of the $T$-action on Sub$^p_G$ with that of the $T$-action on $G$ and characterise 
the same in terms of compactness of the closed subgroup generated by $T$ in the group $\Aut(G)$, in case $G$ is not a vector group. 
We also extend these results to the actions of subgroups of automorphisms. We show that any $n$-torus $G$, $n\geq 2$, more 
generally, any connected Lie group $G$ whose central torus has dimension at least 2, does not admit any automorphism 
which acts expansively on Sub$^p_G$. Our results generalise some results on distal actions by Shah and Yadav, and by  
Chatterjee and Shah, and some results on expansive actions by Prajapati and Shah.
\end{abstract}

\maketitle

\section{Introduction}
Distal and expansive actions are two significant areas of study in topological dynamics. The notion 
of distality was introduced by David Hilbert to study non-ergodic actions on compact spaces. Distal actions have been studied 
by Ellis \cite{E} and Furstenberg \cite{Fu} on compact spaces, Abels \cite{Ab1, Ab2} and Moore \cite{Mo} on Lie groups, 
and Raja-Shah \cite{RSh1, RSh2} and Shah \cite{Sh1} on locally compact groups, see also the references cited in 
\cite{ShY3} and \cite{CSh}. The notion of expansivity was introduced by Utz \cite{U} to study chaotic orbits. 
Since then it has been widely studied by many in different contexts; see Gl\"ockner-Raja \cite{GR}, Prjapati-Shah \cite{PrSh}, 
Shah \cite{Sh2} and Chodhuri et al \cite{ChoFY} and see also the references cited in \cite{PaPrSh}.

For a Hausdorff topological space $X$, a homeomorphism $T$ of $X$  
is said to be {\it distal} (equivalently, $T$ acts {\it distally} on $X$) if for any pair of distinct 
elements $x,y \in X$, the closure of the double orbit 
$\{(T^n(x), T^n(y)) \mid n\in \Z\}$ in $X \times X$ does not intersect 
the diagonal, i.e.\ for $x,y \in X$ with $x \neq y $, 
$\ol{\{(T^n(x), T^n(y))\mid n \in \Z\}} \cap \{(d,d)\mid d \in X\}=\emptyset$.
 If $X$ is a compact metric space with the metric $d$, then $T$ is distal if and only if for every pair 
  $x,y \in X$ with $x \neq y$, $\mathrm{inf}\{ d(T^n(x), T^n(y)) \mid n \in \Z\} > 0.$ If $X$ is a topological
   group and $T$ is an automorphism, then $T$ is distal if and only if for every $x \in G$ with $x \neq e$, 
   $e \notin \ol{\{T^n(x)\} \mid n \in \Z\} }$. 

Let $X$ be a metrizable topological space with a metric $d$, a homeomorphism $T$ of $X$ is said to be {\it expansive} 
if there exists $\epsilon> 0$ such that the following hold: If $x,y \in X$ with $ x \neq y$, then
  $\sup\{d(T^n(x), T^n(y))\mid n\in\Z\}> \epsilon$. Here, $\epsilon$ is called an expansive constant for $T$. It is known that 
  the expansivity of a homeomorphism on any compact metric space is independent of the metric. For a locally compact 
  topological group $G$, an automorphism $T$ is expansive if $\cap_{n \in \Z}T^n(U)= \{e\}$ for some neighbourhood 
   $U$ of $e$. In particular, if $T$ is expansive on $G$, then $G$ is first countable and admits a left invariant metric 
   and this definition agrees with the one given above for the metric space. It is also well known that for any infinite compact 
   metric space, the class of distal homeomorphisms and the class of expansive homeomorphisms are mutually disjoint 
    \cite[Theorem 2]{Bry}. 

For a locally compact Hausdorff topological group $G$, let $\Sub_G$ denote the set of all closed
 subgroups of $G$ equipped with the Chabauty topology. Note that $\Sub_G$ is compact and Hausdorff, and 
 if $G$ is second countable then it is metrizable. Let $\Aut(G)$ denote the group of all automorphisms of $G$
   (i.e.\ homomorphisms of $G$ which are also homeomorphisms). There is a natural action of $\Aut(G)$ on $\Sub_G$; 
   namely,  
   $$(T,H) \mapsto T(H),\ T \in \Aut(G),\ H \in \Sub_G.$$ 
   This action gives rise to a homomorphism from $\Aut(G)$
   to $\mathrm{Homeo}(\Sub_G)$. Motivated by the fact that the image of $\Aut(G)$ yields a large class of 
   homeomorphisms of $\Sub_G$, we aim to explore the dynamics of this subclass, 
   particularly in the context of distality and expansivity. 

Let $G$ be a connected Lie group and $T \in \Aut(G) $. We say that $T$ acts distally (resp.\ expansively) on 
$\Sub_G$ if the homeomorphism of $\Sub_G$ corresponding to $T$ is distal (resp.\ expansive). 
Shah and Yadav \cite{ShY3} first studied such distal actions for connected Lie group $G$ and they showed that for a large
  class of Lie groups $G$ without any non-trivial compact connected central subgroups, an automorphism 
  $T$ acts distally on $\Sub^a_G$, the space of closed abelian subgroups of $G$, if and only if $T$ generates
  a relatively compact subgroup in $\Aut(G)$. To prove this characterisation, they first showed that for  
 such a $G$, $T$ acts distally on $\Sub^a_G$ implies that $T$ is distal 
  (on $G$); see Corollary 3.7 in \cite{ShY3}. Shah and Prajapati \cite{PrSh} first studied such expansive actions for locally 
  compact second countable groups. They showed that  if $T$ acts expansively on $\Sub_G$, then $T$ is expansive on $G$, 
  moreover a (non-trivial) connected Lie group $G$ does not admit any automorphism which acts expansively on $\Sub^a_G$.  
  Shah, together with Palit \cite{PaSh}, and with Palit and Prajapati \cite{PaPrSh}, investigated the distal and expansive actions of 
  automorphisms on $\Sub^a_G$ for discrete groups $G$, where $G$ is either polycyclic or a lattice in a connected Lie group.  
  Note that $\Sub^a_G$ is very large for many groups $G$, in particular if $G$ is abelian, then it is the same as $\Sub_G$. 

A natural question that arose from the above investigations was whether distal (resp.\ expansive) actions of automorphisms of $G$ 
can be characterised (resp.\ exist) on other smaller invariant subspace of $\Sub^a_G$. A recent work by Chatterjee and Shah 
\cite{CSh}  considers one such subspace: the class of (non-trivial) smallest closed connected abelian subgroups; namely, closed
 one-parameter subgroup of $G$.  Let $\spg$ denote the smallest closed subset of $\Sub_G$ containing 
 all closed one-parameter subgroups of $G$. Note that $\spg$ is compact and it is invariant under the
 action of $\Aut(G)$. In Theorem 1.1 of \cite{CSh}, the class of distal actions of automorphisms on the 
 space $\spg$ is characterised when $G$ is a Lie group without any central torus as follows; 
 if $G$ is abelian, then $G$ is a vector group, i.e.\ $G$ is isomorphic to $\R^n$ for some $n\in\N$, and $T$ acts 
 distally on $\spg$ if and only if $T \in \K\D$, where $\K$ is a compact subgroup of $\GL(n,\R)$  and $\D$ is the center of 
 $\GL(n,\R)$. If $G$ is not abelian, then $T$ acts distally on $\spg$ if and only if $T \in \K$, a compact subgroup 
 of $\Aut(G)$. This was also generalised to  hold for a class of automorphisms $T$ which act distally on the 
 maximal central torus, see Theorem 1.4 of \cite{CSh}.
 
 We know that the maximal compact connected central subgroup of a connected Lie group (maximal central torus) is either 
 trivial or isomorphic to $\mT^n$, the $n$-torus (also known as the central torus of $G$), for some $n\in\N$. The work 
 of Chatterjee and Shah \cite{CSh} does not cover the case when $G$ is a torus, or the more general case of connected 
 Lie groups $G$ without any condition on the $T$-action on the central torus. Moreover, the behaviour of distal 
 and expansive actions of automorphisms of $\mT^n$ on $\spp_{\mT^n}$ has not yet been explored. In this paper, 
 we investigate these dynamical properties when $G = \mT^n$. We get the following which characterises distal actions on $\spt$. 

\medskip
\noindent {\bf \Cref{distal-t}.}
\emph{Let $G$ be an $n$-torus for some $n\in\N$, and let $T$ be an automorphism of $G$. Then $T$ acts distally on 
$\spg$ if and only if $T^m=\Id$, the identity map, for some $m\in\N$.}

\medskip
The following relates the distal action of an automorphism of $G$ on $\spg$ and distality of the automorphism and it generalises 
Theorem 1.2 of \cite{CSh}. 

\medskip
\noindent {\bf \Cref{distal-g}.}
\emph{Let $G$ be a connected Lie group which is not a vector group and let $T \in \Aut(G)$. If $T$ acts distally on $\spg$, 
then $T$ acts distally on $G$.}

\medskip
The condition that $G$ is not a vector group is necessary in \Cref{distal-g}, for if $G$ is a vector group, 
i.e.\  $G\cong\R^n$ for some $n\in\N$, then for any $\alpha\ne 0$, $\alpha\Id$ acts trivially, and hence, distally 
on $\spg$, (see also \cite[Theorem 1.1\,(1)]{CSh}), but it does not act distally on $G$ if $\alpha\ne \pm 1$. 

For a Lie group $G$, let $\Aut(G)$ be endowed with the compact-open topology. It is a Lie group which is identified 
with a closed subgroup of the Lie algebra automorphisms of the Lie algebra of $G$ (see e.g.\  \cite{DSh}). Let $\Sub^c_G$ 
denote the set of all discrete cyclic subgroups of $G$. Using \Cref{distal-t} on the distal actions of automorphisms on 
$\Sub^p_{\mT^n}$, we obtain a characterisation of distal actions of automorphisms on the space $\spg$ for any connected 
Lie group $G$ in the following.

\medskip
\noindent {\bf \Cref{gen-Lie}.}
\emph{Let $G$ be a connected Lie group and let $T \in \Aut(G)$. Consider the following statements:
\begin{enumerate}[label=\upshape(\roman*), leftmargin=*, widest=iii]
\item[$(1)$] $T$ acts distally on $\spg$.
\item[$(2)$] $T$ acts distally on $\ol{\Sub^c_G}$. 
\item[$(3)$] $T$ acts distally on $\Sub^a_G$. 
\item[$(4)$] $T$ acts distally on $\Sub_G$. 
\item[$(5)$] $T$ is contained in a compact subgroup of $\Aut(G)$. 
\end{enumerate}
Then $(2-5)$ are equivalent. If $G$ is not a vector group, then $(1-5)$ are equivalent.}
 
 \medskip
Theorem \ref{distal-t} generalises several results from the work of Chatterjee and Shah \cite{CSh}, notably 
Theorems 1.1\,(2), 1.2 and 1.4 of \cite{CSh}. Note that when $G$ is a vector group, i.e.\ $G=\R^n$, for some $n\in\N$, 
Theorem 1.1\,(1) of \cite{CSh} has characterised the class of automorphisms of $G$ which act distally on $\spg$. 
We also generalise Theorem 1.5 of \cite{CSh} about the action of a subgroup of $\Aut(G)$ on $\spg$, see Theorem \ref{subgp-action}.

For expansive actions, it is known that $\mT^n$ does not admit any automorphism which acts expansively on $\Sub_{\mT^n}$  
\cite[Theorem 3.2]{PrSh}. The question arises whether any automorphism of $\mT^n$ acts expansively on a smaller 
invariant subspace of $\Sub_{\mT^n}$. The following shows that in case of $\spt$ the answer is negative. 

\medskip
\noindent {\bf \Cref{expansive-t}.}
\emph{Let $G=\mT^n$, the $n$-torus, for any $n\geq 2$. Then $G$ does not admit any automorphism that acts expansively 
on $\spg$.}

 \medskip
As a consequence of the above theorem, we get the following for a larger class of connected Lie groups. 

\medskip
\noindent {\bf \Cref{Lie-t}.}
\emph{Let $G$ be a connected Lie group such that it contains a central torus of dimension at least $2$. Then $G$ does not 
admit any automorphism that acts expansively on $\spg$.}

\medskip
 \Cref{Lie-t} generalises Theorem 3.1 of \cite{PrSh}, in case $G$ is a connected Lie group as above.  

For many groups $G$, compact spaces $\Sub_G$, $\Sub^a_G$ and $\Sub^c_G$ are identified (see Baik and Clavier \cite{BC1, BC2}, 
Bridson et al.\  \cite{BrHK}, Hamrouni and Kadri \cite{HaKa}, and also Pourezza and Hubbard \cite{PHu}); one can also identify 
$\Sub^p_G$ for some groups $G$, e.g.\  $\Sub^p_{\R^n}\mi\{0\}$ is homeomorphic to $\R\mathbb{P}^{n-1}$, the real projective 
space of dimension $n-1$, $n\geq 2$. The study of the action of $\Aut(G)$ on $\Sub_G$ and on its closed (compact) invariant 
subspaces leads to a better understanding of dynamics on these spaces.

For a subgroup $H$ of $G$, let $H^0$ denote that connected component of the identity $e$ in $G$ and $\ol{H}$ is the closure 
of $H$ in $G$; both $H^0$ and $\ol{H}$ are subgroups of $G$. A $k$-torus $\mT^k$ is a compact connected abelian Lie group 
of dimension $k$, $k\in\N$. Any compact abelian subgroup is divisible if and only if it is connected (i.e.\ it is a torus). A maximal 
central torus (a maximal compact connected central subgroup) of $G$ is characteristic in $G$. Any one-parameter subgroup 
in $G$ is a continuous homomorphism from $\R$ to $G$. It is either closed and isomorphic to $\R$, $\{e\}$ or a $1$-torus 
$\mT\cong\Sc^1$ (the unit circle), or its closure is a $k$-torus for some $k\geq 2$ \cite[Lemma 2.3]{CSh}. We may denote 
a one-parameter subgroup $\{x_t\}_{t\in\R}$ by just $\{x_t\}$. 

In \S\,2 we discuss the topology of $\Sub_G$ and the action of $\Aut(G)$ on $\Sub_G$. We also derive some properties of $\spt$ 
and prove some useful lemmas about its structure. In \S\,3 we prove results on distal actions of automorphisms $T$ on $\spt$, and 
also that of automorphism groups of $\mT^n$ on $\spt$ and its consequences for the action on $\spg$ for a connected Lie group $G$. 
In \S\,4 we prove a result on the orbits of subspaces of $\R^n$ and prove the main result about expansive actions. 

\section{Structure and properties of $\Sub_G$ and $\spt$} 

For any locally compact (Hausdorff) group $G$, the space $\Sub_G$, of all closed subgroups of $G$ is endowed with 
the Chabauty topology, which is generated by a sub-basis 
$$
\{\U_1(K)\mid K\subset G \mbox{ is compact}\}\cup\{\U_2(U)\mid U\subset G \mbox{ is open}\},$$  
where $\U_1(K)=\{H\in\Sub_G\mid H\cap K=\emptyset\}$ and $\U_2(U)=\{H\in\Sub_G\mid H\cap U\neq\emptyset\}$. 
Note that $\Sub_G$ is compact and Hausdorff, and it is metrizable if $G$ is second countable  \cite[Lemma E.1.1]{BP}. 
Since we deal with closed subgroups of a connected Lie group $G$, we have that $\Sub_G$ is metrizable. 
The following criteria for convergence in $\Sub_G$ is well-known (see e.g.\ Proposition E.1.2 in \cite{BP}). 

\begin{lemma}\label{conv}
Let $G$ be a connected Lie group. A sequence 
$\{H_n\}\subset \Sub_G$ converges to $H\in \Sub_G$ 
if and only if the following hold:
\begin{enumerate}[label=\upshape(\roman*), leftmargin=*, widest=iii]
\item[${\rm(I)}$] For $g\in G$, if there exists a subsequence $\{H_{n_k}\}$ of 
$\{H_n\}$ with $h_k\in H_{n_k}$, $k\in\N$, such that $h_k\to g$ in $G$, then $g\in H$. 
\item[${\rm(II)}$] For every $h\in H$, there exists a sequence $\{h_n\}_{n\in\N}$ such that
 $h_n\in H_n$, $n\in\N$, and $h_n\to h$.
\end{enumerate}
\end{lemma}

\begin{remark} \label{rem0} It is known that for a connected Lie group $G$, $\Sub^a_G$, the set of closed abelian subgroups is closed 
(and hence compact) in $\Sub_G$ and it follows easily from \Cref{conv}\,(II) above. But $\Sub^c_G$, the set of discrete 
cyclic subgroups is not closed in $\Sub_G$ in case $G$ is non-trivial, in fact $\spg\subsetneq \ol{\Sub^c_G}$. Any 
non-trivial closed one-parameter subgroup $H:=\{x_t\}_{t\in\R}$ in $G$ is a limit of discrete cyclic subgroups 
$H_n:=\{x_{mt/n}\mid m\in\Z\}$, $n\in\N$, of $H$, where $t$ is nonzero and if $H$ is bounded, then $t$ is the smallest 
positive number such that $x_t=e$ (see also \cite[p.\ 154]{CSh}). Thus $\spg\subset\ol{\Sub^c_G}$. However, they are 
not equal as $\spg$ does not contain any non-trivial discrete cyclic subgroups \cite[Lemma 2.2]{CSh}. 
\end{remark}

 Recall that for a connected Lie group $G$, there is a natural action of the space $\Aut(G)$ of (bi-continuous) automorphisms 
 of $G$ on $\Sub_G$; namely, the map $H\mapsto T(H)$, $H\in\Sub_G$, $T\in\Aut(G)$. This is a continuous group action on 
 $\Sub_G$ by homeomorphisms and it keeps the following subspaces invariant: $\Sub^a_G$ consisting of closed abelian 
 subgroups, $\Sub^c_G$ consisting of discrete cyclic subgroups and its closure, the space of closed one-parameter subgroups 
 and its closure $\Sub^p_G$. For some basic structural properties of $\Sub_G$ and the action of $\Aut(G)$ on it, 
 we refer the readers to \cite{ShY3}. 

Note that $\Aut(\mT^n)\cong \GL(n, \Z),$ where $\GL(n, \Z)$ denotes the group of invertible $n \times n$ matrices of determinant 
$\pm 1$ with integer entries. The compact open topology on $\GL(n, \Z)$ is discrete. Therefore, any compact subgroup 
of $\GL(n, \Z)$ is finite, in particular, an element of $\Aut(\mT^n)$  generates a relatively compact subgroup in $\Aut(\mT^n)$ 
if and only if it has finite order. 

Any one-parameter subgroup in a torus $\mT^n$ is either closed and isomorphic to the trivial subgroup or to $\mT$, or its closure 
is isomorphic to a $k$-dimensional torus for $1<k\leq n$ when $n\geq 2$. Moreover, $\spt$ contains all $k$-tori, $1\leq k\leq n$ and 
the trivial subgroup \cite[Lemma 2.3]{CSh}.  In fact, since $\mT^n$ is compact and any torus is divisible, 
any limit of a sequence consisting of sub-tori in $\Sub_{\mT^n}$ is also divisible, hence connected, and hence it is a torus.   
Note that for $n\geq 2$, $\spt$ is infinite; i.e.\ the set of closed one-parameter subgroups is infinite as the set of roots of unity 
is dense in $\mT^n$, and $\mT^n$ is exponential. Note that $\spt$ is countable as any $k$-torus in $\mT^n$ corresponds to 
a vector subspace $V$ of dimension $k$ in the covering group $\R^n$ (of $\mT^n$) such that $V$ is generated by $k$ 
(linearly independent) elements of $\Z^n$.

For the sake of convenience, we state these facts as a lemma about the structure of $\spt$ which is essentially known.

\begin{lemma} \label{subptn} For an $n$-torus $\mT^n$, the following holds. 
\begin{enumerate}[label=\upshape(\roman*), leftmargin=*, widest=iii]
\item $\spt=\{\mT^k\mid 1\leq k\leq n\}\cup\{\{e\}\}$.
\item $\spt$ is countable.
\item The trivial subgroup $\{e\}$ is isolated in $\spt$.
\end{enumerate}
\end{lemma}

Lemma \ref{subptn}\,(iii) is known (see e.g.\ Lemma 2.2 in \cite{CSh}).

\smallskip
The set $\Sub^p_\mT$ has only two elements, the trivial subgroup $\{e\}$ and $\mT$. Moreover $\Aut(\mT)$ consists of 
only two elements; namely, the identity map and the map $x\mapsto x^{-1}$, and the action of either of these maps on 
$\Sub^p_\mT$ is distal as well as expansive. If $n\geq 2$, $\mT^n$ has infinitely many subgroups which are isomorphic to $\mT^k$ 
for each $k$, $0<k<n$, as observed above. Since $\mT^n$ is compact, the Chabauty topology on $\Sub_{\mT^n}$ is induced 
by the Hausdorff metric \cite[Proposition E.1.3]{BP}, and it allows us to use results from the theory of compact abelian 
metric groups related to the Hausdorff metric. We now state a useful lemma by Berend \cite{Be}. 

\begin{lemma} {\rm \cite[Lemma 4.7]{Be}} \label{Berend1}
Let $G$ be a compact abelian metric group and let $\Gamma$ be the dual group of $G$. A sequence $\{G_m\}_{m=1}^\infty$ of closed 
subgroups of $G$ satisfies $G_m \to G$ (in the Hausdorff metric) if and only if for every nonzero $ \gamma \in \Gamma$ we have 
$\gamma \notin \mathrm{Ann}(G_m)$ for sufficiently large $m$ (where $\mathrm{Ann}(H)$ denotes the annihilator in $\Gamma$ 
of a closed subgroup $H$ of $G$). 
\end{lemma}

The following is a direct consequence of \Cref{Berend1} (\cite[ Lemma 4.7]{Be}).

\begin{lemma}\label{Berend2}
Let $\mT^n$ be the $n$-torus and let $\{G_m\}$ be a sequence of closed subgroups in $\mT^n$. Then $\{G_m\}$ 
does not converge to $\mT^n$ in $\spt$ if and only if there exists a character $\gamma$ of $\mT^n$ such that 
$\gamma \in \mathrm{Ann}(G_m)$ for infinitely many $m$.
\end{lemma}

We now note that the set of $k$-dimensional subspaces in $\R^n$ is closed in $\Sub_{\R^n}$.
Let $\{V_m\}$ be a sequence of vector subspaces with $\dim(V_m)=k$, $m\in\N$ and let $V_m\to V$ in $\Sub_{\R^n}$. We may assume 
that $n>1$ and $0<k<n$. Since each $V_m$ has an orthonormal basis  which can be extended to an orthonormal basis of $\R^n$, 
we can find a sequence of orthogonal matrices $\phi_m\in \SO(n,\R)$ such that $\phi_m(V_1)=V_m$, $m\in\N$. Since $\SO(n,\R)$ is 
compact, we get that $\{\phi_m\}$ has a limit point (say) $\phi$, and hence $\phi(V_1)=V$. Then $\dim(V)=\dim(V_1)=k$. 

We know that $\{e\}$ is isolated in $\spt$. We now show that any proper sub-torus is isolated 
in the set of sub-tori of the same dimension, or more generally, in the space of sub-tori of the same or higher dimension. 
For $1\leq k<n$, let $\mfS_k$ be the space of all sub-tori of dimension $k$ in $\mT^n$ and let 
$\mfH_k=\cup_{m=k}^{n-1}\mfS_m\cup \{\mT^n\}$. Then $\mfH_k\subset\mfH_{k-1}$ for $2\leq k< n$ and 
$$
\spt=\mfH_1\cup \{\{e\}\}=\cup_{k=1}^{n-1}\mfS_k\cup\{\mT^n\}\cup \{\{e\}\}.$$ 
It is easy to see that all $\mfS_k$ and $\mfH_k$ are $T$-invariant for every $T\in\GL(n,\Z)$,
for $1\leq k\leq n-1$. As the set of $k$-dimensional subspaces in $\R^n$ is closed in $\Sub_{\R^n}$, we get 
that each $\mfH_k$ is closed in $\spt$.

\begin{lemma} \label{isol} 
Let the notation be as above. For $n\geq 2$, $\mfH_k$ is closed  (compact) in $\spt$. 
Moreover, every $H\in \mfS_k$ is isolated in $\mfH_k$.
 \end{lemma}
 
\begin{proof} Note that since any sub-torus is divisible and $\mT^n$ is compact, the limit of a sequence of sub-tori is a sub-torus 
or $\mT^n$ (see also \Cref{subptn}\,(i)). Let $\pi:\R^n\to\mT^n$ be the natural projection with $\ker\pi=\Z^n$.  Then every 
$H\in \mathfrak{S}_k$ corresponds a $k$-dimensional vector subspace $V$ of $\R^n$ such that $\pi(V)=H$. As noted above, 
the set of all $k$-dimensional subspaces of $\R^n$ is closed in $\Sub_{\R^n}$, and we get that if $H_m\to H$ in $\spt$ for 
$\{H_m\}\subset \mathfrak{S}_k$, then $\dim(H)\geq k$. Hence $H\in \mfH_k$. In particular,  
$\mfH_k=\cup_{l=k}^{n-1}\mfS_l\cup\{\mT^n\}$ is closed. Thus the first assertion holds.
 
 \medskip
\noindent{\bf Step 1:}  Now we prove the second assertion. We first consider $k=n-1$. Let $H\in \mfS_{n-1}$. We want to show that 
 $H$ is isolated in  $\mfH_{n-1}$. Suppose $H_m\in \mfH_{n-1}$, $m\in\N$, is such that $H_m\to H$, then we may assume that 
 $H_m\in \mfS_{n-1}$ for all $m\in\N$. If possible, suppose $H_m\neq H$ for infinitely many $m$. Passing to a subsequence 
 if necessary, we may assume that $H_m\neq H$ for all $m$. By Lemma 4.7 of \cite{Be} (see also Lemmas \ref{Berend1} 
 or \ref{Berend2}), there exists a nonzero character $\psi$ on $\mT^n$ such that $\psi(H_m)=0$, i.e.\  
 $H_m\subseteq (\ker\psi)^0$, for infinitely many $m$. Since $\dim(H_m)=n-1=\dim(\ker\psi)^0$, we get that 
 $H_m=(\ker\psi)^0$ for infinitely many $m$. Thus $\{H_m\}$ has a constant subsequence, and as 
 $H_m\to H$, we get that $H_m=H$ for infinitely many $m$; this leads to a contradiction. Thus $H_m=H$ for all large $m$. 
 Therefore, each $H\in \mathfrak{S}_{n-1}$ is isolated in $\mfS_{n-1}$, and hence in $\mfH_{n-1}$. In particular, 
 the second assertion holds for $n=2$. 
  
  \medskip
  \noindent{\bf Step 2:} Now suppose $n\geq 3$. Suppose the second assertion holds for every torus with dimension $n-1$. 
  Now we prove the assertion for $\mT^n$. Suppose $1\leq k \leq n-1$. If $k=n-1$, then the assertion follows from Step 1. 
  Now suppose $1\leq k\leq n-2$. Let $H\in\mfS_k$. Suppose $H_m\to H$ for some $\{H_m\}\subset \mfH_k$. 
   If possible, suppose $H_m\ne H$ for infinitely many $m$. Passing to a subsequence if necessary, we may 
   assume that $H_m\ne H$ for all $m$. Since $H\ne\mT^n$, arguing as in Step 1 using Lemma 4.7 of \cite{Be}, we get that 
  there exists a character $\gamma$ of $\mT^n$ such that $\gamma(H_m)=0$, i.e.\ $H_m\subset (\ker\gamma)^0=K$ for infinitely 
  many $m$. Hence it follows that $H\subset K$. Since $\dim(K)=n-1$, by induction, $H$ is isolated in the set (say) $\mfH_k(K)$ 
  of all sub-tori of $K$ with dimension greater than or equal to $k$, and hence $H_m=H$ for infinitely many $m$, which leads 
  to a contradiction. Therefore, $H$ is isolated in $\mfH_k$, and the assertion holds by induction for all $n\geq 2$. 
  \end{proof}

\section{Distal actions of automorphisms of connected Lie groups $G$ on Sub$^p_G$} 

In this section, we first prove \Cref{distal-t} which extends Theorem 1.1\,(2) of \cite{CSh} to compact 
connected abelian Lie groups. Using Theorem \ref{distal-t} we generalise Theorems 1.1\,(2), 1.2, 1.4 and 1.5 of \cite{CSh}.

\begin{theorem}\label{distal-t} Let $G$ be an $n$-torus for some $n\in\N$, and let $T$ be an automorphism of $G$. 
Then $T$ acts distally on $\spg$ if and only if $T^m=\Id$, the identity map, for some $m\in\N$.
\end{theorem}

\begin{proof}
One way assertion is obvious. We prove the converse by induction on the dimension $n$ of the torus; 
i.e.\ we assume that $T$ acts distally on $\spg$ and show that $T^m=\Id$ for some $m\in\N$. If $G=\mT$, 
the one-dimensional torus, then $\Aut(G)$ has only two elements, hence $T^2 =\Id$. Thus the assertion holds trivially 
for $n=1$. Now suppose the assertion holds for any torus $G$ with $\dim G\leq n-1$, where $n\ge 2$. Let $G=\mT^n$. 
Let $H\in\mfS_{n-1}$. Then $H\in\spg$ by \Cref{subptn}\,(i) (see also Lemma 2.3 in \cite{CSh}). As $\spg$ is 
compact, there exists a strictly increasing sequence $\{m_k\}$ in $\N$, such that $T^{m_k}(H)\to L$ for some $L\in\spg$.  
As $H\in\mfS_{n-1}\subset\mfH_{n-1}$, and $\mfH_{n-1}$ is compact and $T$-invariant, we have that $L\in\mfH_{n-1}$. 
Since $T$ acts distally on $\spg$ and $T(G)=G$ and $H\ne G$, it follows that  $L\ne G$. Therefore, $L\in\mfS_{n-1}$. 
By \Cref{isol}, $L$ is isolated in $\mfH_{n-1}$, and hence  $T^{m_k}(H)=L$ for all large $k$. Thus for some $k\in\N$, 
$T^{m_k}(H)=T^{m_{k+i}}(H)$ for all $i\in\N$. For $l=m_{k+1}-m_k$, we have that $H$ is $T^l$-invariant. Since $T^l$ 
acts distally on $\spg$, and also on $\spp_H$, and $\dim H=n-1$, by the induction hypothesis we get that $(T|_H)^{lr}=\Id$ 
for some $r\in\N$. Moreover, since $G/H$ is isomorphic to  $\mT$, we have that $\T^2=\Id$, where $\T$ is the 
automorphisms of $G/H$ induced by $T$. Let $m=2lr$. Then $(T|_H)^m=\Id$ and $\T^m=\Id$. Note that $\Aut(G)$ 
is identified with a subgroup of $\Aut(\gf)$, the group of Lie algebra automorphisms of the Lie algebra $\gf$ of $G$. 
Let $\hf$ be the Lie sub-algebra of $\gf$ which is the Lie algebra of $H$. Then $\gf/\hf$ is the Lie algebra of $G/H$. 
Now (under the identification) $T^m$ (resp.\ $\T^m$) acts trivially on $\hf$ (resp.\ the quotient Lie algebra $\gf/\hf$), 
and hence $T^m$ is unipotent, i.e.\ all its eigenvalues are equal to $1$.  Since $T^m$ acts distally on $\spg$, by 
Theorem 1.3 of \cite{CSh}, $T^m = \Id$. Thus the assertion holds for all $n$ by induction. 
\end{proof}

We say that a subgroup $\Hc$ of homeomorphisms of $X$ acts distally on a topological space $X$ if for any pair of distinct 
elements $x,y\in X$, 
$$\ol{\{(T(x),T(y))\mid T\in \Hc\}}\cap\{(d,d)\mid d\in X\}=\emptyset.$$ 
Note that if $\Hc$ acts distally on $X$, then every element of $\Hc$ acts distally on $X$. 
The following corollary will be useful in the proof of \Cref{subgp-action}.

\begin{corollary} \label{subgp-t} Let $\Hc$ be a subgroup of $\Aut(\mT^n)$ for some $n\in\N$. Then the following are equivalent: 
\begin{enumerate}[label=\upshape(\roman*), leftmargin=*, widest=iii]
\item[$(1)$] Every element of $\Hc$ acts distally on $\spt$. 
\item[$(2)$] $\Hc$ acts distally on $\spt$. 
\item[$(3)$] $\Hc$ is a finite group.
\end{enumerate}
\end{corollary}

\begin{proof} 
The statements $(3)\implies (2)\implies (1)$ are obvious. Now we show that $(1)\implies (3)$. This hold for $n=1$ as
$\Aut(\mT)$ is a group of order $2$. Now suppose $n\geq 2$. Suppose (1) holds. By Theorem \ref{distal-t}, every element of 
$\Hc$ has finite order. Note that $\Aut(\mT^n)=\GL(n,\Z)$ is a discrete subgroup of $\GL(n,\R)$ as well of $\GL(n,\C)$. Then 
$\Hc$ is closed in $\GL(n,\C)$ and by Theorem 1.1 of \cite{FP}, $\Hc$ is compact. As $\Hc\subset\GL(n,\Z)$ is discrete, we 
get that $\Hc$ is finite. Thus (3) holds, and hence $(1-3)$ are equivalent. 
\end{proof} 
 
 In \cite{CSh}, for $T\in\Aut(G)$ it is shown that $T$ acts distally on $\spg$ implies that it acts distally on $G$, where $G$ is 
a connected non-abelian Lie group without any (non-trivial) central torus; more generally, if $G$ is not a vector group and $T$ acts 
distally on the maximal central torus. The following corollary generalises Theorem 1.2 of \cite{CSh}.  As noted earlier, 
the condition in the corollary that $G$ is not a vector group is necessary. 
 
\begin{corollary}\label{distal-g}
Let $G$ be a connected Lie group which is not a vector group and let $T \in \Aut(G)$. If $T$ acts distally on $\spg$, 
then $T$ acts distally on $G$.
\end{corollary}

\begin{proof} Suppose $T$ acts distally on $\spg$ for a connected Lie group $G$ which is not a vector group. 
Let $M$ be the largest compact connected central subgroup of $G$. If $M$ is trivial, then the assertion follows from 
Theorem 1.1\,(2) of \cite{CSh}. Now suppose $M$ is non-trivial. Then $M \cong \mT^n$ for some $n\in\N$. Note that 
$M$ is $T$-invariant and $T|_M$ acts distally on $\spp_M$. By Theorem \ref{distal-t} we get that some power 
of $T|_M$ is the identity map. In particular, $T$ acts distally on $M$. As $T$ acts distally on $\spg$, 
by Theorem 1.2 of \cite{CSh},  $T$ acts distally on $G$. 
\end{proof}

The following theorem generalises Theorems 1.1\,(2) and 1.4 of \cite{CSh}, and also a part of Theorem 4.1 of \cite{ShY3}.

\begin{theorem} \label{gen-Lie}
Let $G$ be a connected Lie group and let $T \in \Aut(G)$. Consider the following statements:
\begin{enumerate}[label=\upshape(\roman*), leftmargin=*, widest=iii]
\item[$(1)$] $T$ acts distally on $\spg$.
\item[$(2)$] $T$ acts distally on $\ol{\Sub^c_G}$. 
\item[$(3)$] $T$ acts distally on $\Sub^a_G$. 
\item[$(4)$] $T$ acts distally on $\Sub_G$. 
\item[$(5)$] $T$ is contained in a compact subgroup of $\Aut(G)$. 
\end{enumerate}
Then $(2-5)$ are equivalent. If $G$ is not a vector group, then $(1-5)$ are equivalent. 
\end{theorem}

\begin{proof}
Statements $(5) \implies (4) \implies (3) \implies (2)$ are trivial.  Also, $(2)\implies (1)$ as $\spg\subset\ol{\Sub^c_G}$ 
(see \Cref{rem0}).  Suppose $G$ is not a vector group. We need to prove that $(1) \implies (5)$. 
Suppose $(1)$ holds, i.e.\ $T$ acts distally on $\spg$. Let $M$ be the largest compact connected central subgroup of $G$. 
Then $M$ is characteristic in $G$, i.e.\ $T(M)=M$. Then (1) implies that $T$ acts distally on $\spp_M$. By Theorem \ref{distal-t}, 
for some $m\in\N$, $T^m$ acts trivially on $M$, in particular, $T$ acts distally on $M$. Now by Theorem 1.4 of \cite{CSh}, 
$T$ is contained in a compact subgroup of $\Aut(G)$. Thus $(5)$ holds, and $(1-5)$ are equivalent if $G$ is not a vector group. 

Now suppose $G$ is a vector group, i.e.\ $G=\R^n$ for some $n\in\N$. Suppose (2) holds. Then (2) implies that 
$T\in{\rm (NC)}$, i.e.\ the trivial subgroup $\{e\}$ is not a limit point of $\{T^k(C)\}_{k\in\Z}$ for any non-trivial discrete cyclic 
subgroup $C$ of $G$. Then by Theorem 4.1 of \cite{ShY3}, $T$ is contained in a compact subgroup of $\Aut(G)$, 
i.e.\ (5) holds, and $(2-5)$ are equivalent if $G$ is a vector group. 
\end{proof}

 The following generalises a part of Theorem 4.1 of \cite{ShY3} and Theorem 1.5 of \cite{CSh} when $G$ is not a vector group. 
 Note that in case $G$ is a vector group, the equivalence of $(1)$ and $(2)$, as well as that of $(3-6)$ is shown in 
 Theorem 1.5 of \cite{CSh}. Note also that it is not possible to generalise Theorem 4.1 of \cite{ShY3} for all connected 
 Lie groups $G$.  For an $n$-torus $G$, $n\geq 2$, $\Aut(G)\cong\GL(n,\Z)$ and every $T\in\Aut(G)$ belongs to (NC). 
 If $T\in\Aut(G)$ and $C$ is a non-trivial discrete cyclic subgroup of $G$, then $C$  is finite and it is contained in $A_m$, 
 the group of $m$th roots of unity in $G$ for some $m\in\N$. Since $A_m$ is finite and $T$-invariant, $T^l$ acts trivially on 
 $A_m$ for some $l\in\N$, and we get that  $\{T^k(C)\}_{ k\in\Z}$ is a finite sequence consisting of non-trivial cyclic subgroups, 
 and hence it can not have $\{e\}$ as a limit point. Thus $T\in{\rm (NC)}$. But not every $T$ can belong to a compact (finite) 
 subgroup of $\Aut(G)\cong \GL(n,\Z)$, as there are $T\in\Aut(G)$ which have infinite order, e.g.\  non-trivial unipotent elements in 
 $\GL(n,\Z)$. 

\begin{theorem} \label{subgp-action} Let $G$ be a connected Lie group. Let $\Hc$ be a subgroup of $\Aut(G)$. 
Consider the following statements:
\begin{enumerate}[label=\upshape(\roman*), leftmargin=*, widest=iii]
\item[$(1)$] Every element of $\ol{\Hc}$ acts distally on $\spg$. 
\item[$(2)$] $\Hc$ acts distally on $\spg$.
\item[$(3)$] $\Hc$ acts distally on $\ol{\Sub^c_G}$. 
\item[$(4)$] $\Hc$ acts distally on $\Sub^a_G$. 
\item[$(5)$] $\Hc$ acts distally on $\Sub_G$. 
\item[$(6)$] $\ol{\Hc}$ is a compact group.
\end{enumerate}
\noindent Then $(1)$ and $(2)$ are equivalent, and $(3-6)$ are equivalent. If $G$ is not a vector group, then $(1-6)$ are equivalent. 
\end{theorem}

\begin{proof} It is obvious that $(6)\implies (5) \implies (4)\implies (3) \implies(2)$. Also if $\Hc$ acts distally on $\spg$, then 
so does $\ol{\Hc}$ since $\spg$ is compact \cite[Theorem 1]{E}.  Therefore $(2)\implies (1)$. Now suppose (1) holds.
 Let $M$ be the largest compact connected central subgroup of $G$. Then $M$ is characteristic in $G$ and in particular, 
 it is invariant under the action of $\ol{\Hc}$. If every element of $\ol{\Hc}$ acts distally on $\spp_M$, then by 
 \Cref{subgp-t}, $\{T|_M\mid T\in \ol{\Hc}\}$ is a finite group. Therefore $\ol{\Hc}$, and hence, $\Hc$ acts distally 
on $M$. If $G$ is not a vector group, by Theorem 1.5 of \cite{CSh}, $\ol{\Hc}$ is compact. Thus (6) holds, 
and hence $(1-6)$ are equivalent. 

If $G= \R^n$, a vector group, where $n\in\N$, then the assertions $(1)\implies (2)$ and $(3)\implies (6)$ are proven in 
Theorem 1.5 in \cite{CSh}. 
\end{proof}

\section{Expansivity of actions of automorphisms of $\mT^n$ on $\spt$}

In this section we prove that $\mT^n$ does not have any automorphism that acts expansively on $\spt$ for any $n\geq 2$ 
(see \Cref{expansive-t}), and prove \Cref{Lie-t}. We also prove that every element of $\GL(n,\Z)$ has infinitely many orbits 
consisting of $(n-1)$-dimensional rational subspaces (see \Cref{rational}). 

We first state some well-known properties about expansive maps, as listed in Lemma 2.1 of \cite{PrSh}.

\begin{lemma} {\rm \cite[Corollary 5.22 \& Theorem 5.26]{W}} \label{expa-prop} 
	Let $(X,d)$ be a compact metric space. Then the following hold for homeomorphisms of $X$:
\begin{enumerate}[label=\upshape(\roman*), leftmargin=*, widest=iii]
\item[$(1)$] Expansivity is a topological conjugacy invariant.
\item[$(2)$] Expansivity of a homeomorphism is independent of the metric chosen as long as the metric induces the topology of $X$. 
    However, expansivity constant may change.
    
\!\!\!\!\!\!\!\!\!\!\!\!\!\!\!\!\! Moreover, the following hold for any homeomorphism $\phi$ of $X$:

\item[$(3)$] $\phi^n$ is expansive for some $n\in \Z\setminus \{0\}$ if and only if $\phi^n$ is expansive for all $n\in \mathbb{Z}\setminus \{0\} $.
\item[$(4)$] For any $n\in \Z\setminus \{0\}$, if $\phi$ is expansive then $\phi^n$ has only finitely many fixed points.
\item[$(5)$] If $\phi$ is expansive and $Y$ is a closed $\phi$-invariant subset of $X$, then $\phi|_Y$ is also expansive.
\end{enumerate}
\end{lemma}

Theorem 3.1 of \cite{PrSh} shows that a non-trivial connected Lie group $G$ does not admit any automorphism that acts expansively on 
$\Sub^a_G$. As $\spg\subset\Sub^a_G$, the space of closed abelian subgroups of $G$, the following generalises 
Theorem 3.1 of \cite{PrSh} for the case when $G$ is an $n$-torus, $n\ge 2$. (For $n=1$, $\spg$ has only two elements, 
and hence it holds trivially that every $T\in\Aut(G)$ acts expansively on $\spg$.) 
 
 \begin{theorem}\label{expansive-t}
    Let $G=\mT^n$, the $n$-torus, for any $n\geq 2$. Then $G$ does not admit any automorphism that acts expansively on $\spg$.
 \end{theorem}

 Before proving the theorem, we define a notion of rational subspaces in $\R^n$ and discuss their orbits under the action of 
 $\GL(n,\Z)$. A subspace $W$ of $\R^n$ is said to be a {\it rational subspace}, if $W$ is generated by  $W\cap\Z^n$; 
 equivalently, if $W\cap\Z^n$ is isomorphic to $\Z^k$, where $k=\dim(W)$. Note that $W$ is a rational subspace if and only if 
 the group $W+\Z^n$ is closed in $\R^n$. Each $k$-dimensional rational subspace $W$ corresponds to a $k$-sub-torus ($k$-dimensional 
 compact connected subgroup) in $\mT^n$; namely, the image of $W$ in $\mT^n=\R^n/\Z^n$. Any subspace generated by some 
 integer points of $\R^n$ is a rational subspace. Note  that if $W$ is a rational subspace of $\R^n$, then $T(W)$ is also a rational subspace 
 (of $\R^n$) with the same dimension as that of $W$, for all $T\in\GL(n,\Z)$. There are countably infinitely many $k$-dimensional rational 
 subspaces in $\R^n$ for each $k$ with $0< k< n$. Let $\mathbb{H}_k$ (resp.\ $R_k$) denote the space of all $k$-dimensional subspaces 
 (resp.\ rational subspaces) of $\R^n$. Then $\mathbb{H}_k$ is a closed (compact) proper subspace of $\Sub_{\R^n}$, and $\mathbb{H}_k$ 
 (resp.\ $R_k$) is invariant under the action of $\GL(n,\R)$ (resp.\ $\GL(n,\Z)$). 
 It is easy to see that  $R_k$ is dense in $\mathbb{H}_k$, $0\leq k\leq n$. 
 
A linear map $T\in\GL(n,\R)$, $n\geq 2$, is said to be {\it proximal} if it has a unique eigenvalue of maximum absolute value; such 
an eigenvalue is real. It is well-known that if $T$ is proximal, then for any  $L\in\Sub^p_{\R^n}\mi\{0\}\cong\RP^{n-1}$, $T^n(L)\to L_\alpha$, 
where $L_\alpha$ in $\spp_{\R^n}$ is the one-dimensional eigenspace corresponding to the real eigenvalue $\alpha$ 
with maximum absolute value. 
 
 A linear map $T\in\GL(n,\R)$ is distal (i.e.\ it acts distally on $\R^n$) if and only if all its eigenvalues have 
 absolute value $1$; this is well-known and easy to prove (see e.g.\ \cite{Mo}, \cite{CoGu} or \cite{Ab1}). Moreover, 
 if $T\in\GL(n,\Z)$ is distal, then all its eigenvalues are roots of unity and $T^m$ is unipotent for some $m\in\N$. 
 
 The following proposition will be useful for the proof of \Cref{expansive-t}. Note that the condition in the proposition that $T$ does not keep 
 any non-trivial  proper rational subspace of $\R^n$ invariant is equivalent to the condition that $T$ does not keep any proper sub-torus of 
 $\mT^n$ invariant, and hence it implies that either $T^m=\Id$ for some $m\in\N$ or that $T$ is ergodic on $\mT^n$, $n\geq 2$; this follows 
 from Theorem 2.3 or Theorem 3.15 of \cite{KiSch} (see also Proposition 2.1 of \cite{R} or that of \cite{RSh2}), as there is a connected 
 $T$-invariant subgroup $H$ such that the $T$-action on $H$ is ergodic and the corresponding automorphism $\T$ of $G/H$ is distal. 
 Now the condition implies that $G=H$ and $T$ is ergodic, or $H=\{e\}$ and $T$ is distal, in the latter case $T^m$ is unipotent for some 
 $m\in\N$ \cite[Lemma 2.5]{Ab2}, and hence $T^m=\Id$, otherwise $T$ would keep a proper rational subspace invariant.  We will later 
 see that \Cref{rational} holds for all $T\in\GL(n,\Z)$, $n\geq 2$, without this extra condition on $T$ (see \Cref{rem1}). 
 
 \begin{proposition} \label{rational} Let $T\in\GL(n,\Z)$, $n\ge 2$.  Suppose $T$ does not keep any nonzero proper rational subspace 
 of $\R^n$ invariant. Then there are infinitely many $(n-1)$-dimensional rational subspaces with disjoint $T$-orbits in $\Sub_{\R^n}$. 
  \end{proposition}
  
 \begin{proof} For $n\geq 2$, let $T\in\GL(n,\Z)$ be such that $T$ does not keep any nonzero proper rational subspace of $\R^n$ 
 invariant. This is equivalent to the condition that  any nonzero proper $T$-invariant subspace is not contained in any proper rational 
 subspace of $\R^n$.  If $V$ is a $T$-invariant subspace contained in a proper rational subspace (say) $W$ of $\R^n$, then 
 $V'=(\ol{V+\Z^n})^0$ is a $T$-invariant rational subspace of $\R^n$ and it is contained in $W$. Therefore, $V'$ is proper and 
 the condition on $T$ as in the hypothesis implies that $V'=\{0\}$, and hence that $V=\{0\}$.

 Now we prove the assertion that there are infinitely many $(n-1)$-dimensional rational subspaces with disjoint $T$-orbits. 
 It is easy to see that the assertion holds for $T$ if and only if it holds for $T^m$ for any $m\in\Z\mi\{0\}$. If $T^m=\Id$ 
 for some $m\in\N$, then all the rational subspaces are $T^m$-invariant, i.e.\ the assertion holds for $T^m$, and hence, 
 for $T$. Now suppose $T^m\ne\Id$ for any $m\in\N$. Note that $T$ is not distal. For if $T$ is distal, then  it follows from Lemma 2.5
 of \cite{Ab2} that  for some $m\in\N$, $T^m$ is unipotent, i.e\  all its eigenvalues are equal to $1$, and its eigenspace $V_1$ is 
 $T$-invariant, and $V_1$ is a nonzero proper rational subspace, which leads to a contradiction. Thus $T$ is not distal, i.e.\ it has 
 an eigenvalue with absolute value not equal to $1$. Since $\det(T)=\pm\,1$, $T$ has at least one eigenvalue with absolute 
 value greater than $1$, and another with absolute value less than $1$. Thus one can choose two distinct eigenvalues of $T$ with 
 different absolute values, say, $\alpha$ and $\beta$ with $|\alpha|>|\beta|$. 
 
\medskip
 \noindent{\bf Case I}: Suppose $\alpha$ and $\beta$ as above are real, with $|\alpha|>|\beta|$. Let $V_{\alpha}$ 
 (resp.\ $V_{\beta})$ be the eigenspace corresponding to $\alpha$ (resp.\ $\beta$).  For any proper rational subspace $W$ of $\R^n$, 
 $W\cap V_\alpha$ is a $T$-invariant subspace, and hence we have that $W\cap V_\alpha=\{0\}$. Since this holds in particular 
 for any rational space $W\in R_{n-1}$,  we have that $\dim(V_\alpha)=1$. Similarly, $V_\beta\cap W=\{0\}$ for every 
 $W\in R_{n-1}$ and $\dim(V_\beta)=1$. 
  
 Let $V_{\alpha\beta}=V_{\alpha}+V_{\beta}$. Then $\dim(V_{\alpha\beta})=2$ and $T(V_{\alpha\beta})=V_{\alpha\beta}$.  
Moreover, for any rational space $W\in R_{n-1}$, 
$\dim(W\cap V_{\alpha\beta})=1$ as $V_{\alpha\beta}\not\subset W$. Let $W\cap V_{\alpha\beta}=L_w$. Note that as 
$|\alpha|>|\beta|$, both $T|_{V_{\alpha\beta}}$ and $T^{-1}|_{V_{\alpha\beta}}$ are proximal. Then 
 $$
 T^m(L_w) \to V_\alpha\ \mbox{ and }\  T^{-m}(L_w) \to V_\beta\ \mbox{ as }\ 
 m\to \infty.$$
 Now suppose $T^{m_k}(W)\to H$ for some unbounded sequence $\{m_k\}\subset\Z$. Then 
 $V_\alpha\subset H$ (resp.\ $V_\beta\subset H)$ if $m_k\to \infty$ (resp.\ if $m_k\to -\infty$). 
 Let 
 $$\mathcal{A}=\{H\in\mathbb{H}_{n-1}\mid V_{\alpha}\subset H \ \mbox{ or } \ V_{\beta}\subset H\}.$$ Then $\mathcal{A}$ is a  
 closed subset of $\mathbb{H}_{n-1}$, the space of all $(n-1)$-dimensional subspaces of $\R^n$, and  $\mathbb{H}_{n-1}$ 
 is a proper closed (compact) subset of $\Sub_{\R^n}$.  Any rational subspace $W\in R_{n-1}$  does not belong to $\A$ 
 as it contains neither $V_{\alpha}$ nor $V_{\beta}$. We can show that $\mathcal{A}$ is closed using Lemma 2.2 of \cite{ShY3}. 
 Since $\mathbb{H}_{n-1}$ is a (compact) metric space and $\mathcal{A}$ is a proper closed subset of it, we can choose a neighbourhood 
 $U$ of $\mathcal{A}$ such that $\mathbb{H}_{n-1}\mi U$ has nonempty interior, and hence $\mathbb{H}_{n-1}\mi U$ contains 
 infinitely many rational subspaces from $R_{n-1}$.  Now $U$ contains all but finitely many elements of the orbit $\{T^m(W)\}_{m\in \Z}$ 
 of any rational subspace $W\in R_{n-1}$. Choose $W_1\in R_{n-1}$ such that $W_1\not\in U$ and $T^m(W_1)\in U$  
 for all $m\in\Z$ with $|m|> l_1$ for some $l_1\in\N$. Since $R_{n-1}\mi U$ is infinite, we can choose $W_2\in R_{n-1}$ such that 
 $W_2\not\in U$ and $W_2\neq T^m(W_1)$ if $|m| \leq l_1$. Now $T^m(W_2)\in U$ for all $m\in\Z$ with $|m|> l_2$, for some $l_2\in\N$. 
 Then the $T$-orbit of $W_2$ is disjoint from that of $W_1$. For if $T^m(W_1)=W_2$ for some $m\in\Z$, then $|m|> l_1$, 
 but then $T^m(W_1)\in U$ while $W_2\not\in U$, which leads to a contradiction. 
 
For $k\geq 2$ and $1\leq i\leq k$, suppose there are 
 $W_i\in R_{n-1}$ such that $W_i\not\in U$  with $T^m(W_i)\in U$ for all $m\in\Z$ with $|m|> l_i$ and $W_i\neq T^m(W_j)$ for any 
 $1\leq j< i$ and $|m|\leq l_j$. Since $R_{n-1}\mi U$ is infinite, we can choose a rational subspace $W_{k+1}$  as follows:
 $$
 W_{k+1}\in R_{n-1},\ W_{k+1}\not\in U\ \mbox{ and }\ W_{k+1}\notin\cup_{i=1}^k\{T^m(W_i)\mid  |m|\leq l_i\}.$$ 
Hence the $T$-orbit of $W_{k+1}$ is disjoint from those of $W_j$, $1\leq j\leq k$. Moreover, $T^m(W_{k+1})\in U$ for all $m\in\Z$ with 
$|m|\geq l_{k+1}$ for some $l_{k+1}\in\N$. It is easy to see that the $T$-orbit of $W_{k+1}$ is disjoint from those of $W_1,\ldots, W_k$. 
Thus by induction, there exist infinitely many $(n-1)$-dimensional rational subspaces $W_k$, $k\in\N$, whose $T$-orbits are disjoint. 

   \medskip 
   \noindent{\bf Case II}: Suppose one of the eigenvalues $\alpha$ and $\beta$ of $T$ (as above) is real, and the other is complex. 
   Replacing $T$ by $T^{-1}$ if necessary, we may assume that $\alpha$ is real with $|\alpha|>1$, and $\beta$ is complex with 
   $|\beta|<1$. Let $V_{\alpha}$ be the eigenspace for $\alpha$, and $V_{\beta}$ be the  $2$-dimensional vector subspace 
   of $\R^n$ such that  $T(V_\beta)=V_\beta$, and the eigenvalues of $T|_{V_\beta}$ are $\beta$, $\bar\beta$. Let  
    $V_{\alpha \beta}=V_{\alpha}+V_{\beta}$. Then $\dim(V_{\alpha \beta})=3$, and $T$ keeps $V_{\alpha\beta}$ invariant.  
    There exist $\phi$ and $S$ in $\GL(V_{\alpha\beta})$ such that $T|_{V_{\alpha\beta}}=\phi S=S\phi$, 
    $S|_{V_\alpha}=T|_{V_\alpha}=\alpha\Id$, $S|_{V_\beta}=|\beta|\Id$,  $\phi|_{V_\alpha}=\Id$, $\phi$ keeps $V_\beta$ invariant 
    and $\phi|_{V_\beta}$ is contained in a compact subgroup of $\GL(V_\beta)$. In particular, $\phi$ is contained in a compact 
    subgroup of $\GL(V_{\alpha\beta})$. Since $V_\alpha$ and $V_\beta$ are $T$-invariant, as noted above, 
    neither $V_\alpha$ nor $V_\beta$ is contained in any proper rational subspace. 

For a rational subspace $W\in R_{n-1}$, let $L_w=W\cap V_{\beta}$. 
Then $V_\beta=L_w + L'_w$ for some one dimensional subspace $L'_w$ of $V_\beta$. 
Here, $L'_w\not\subset W$ as $V_\beta\not\subset W$. Let $V_1:=V_{\alpha}+L'_w$. Then $\dim(V_1)= 2$ and 
$V_1$ is $S$-invariant. Now $\dim(W\cap V_1)=1$ as neither $V_{\alpha}$ nor $L'_w$ is contained in $W$.  
Let  $S_1:=S|_{V_1}$. Then $S_1$ has two eigenvalues $\alpha$ and $|\beta|$ and both $S_ 1$ and $S_1^{-1}$ 
are proximal as $|\alpha|> |\beta|$. Let $L_1= W\cap V_1$. Then $S^m(L_1)=S_1^m(L_1)\to V_\alpha$ and 
$S^{-m}(L_1)=S_1^{-m}(L_1)\to L'_w$ in $\Sub^p_{\R^n}$ as $m\to \infty$. Now for $m\in\Z$, 
\begin{eqnarray*}
T^m(W\cap V_{\alpha\beta}) & = & \phi^m S^m(L_w+L_1) \\
& = & \phi^m(S^m(L_w)+S^m(L_1)) \\
& = & \phi^m(L_w)+\phi^m(S^m(L_1)).
\end{eqnarray*}
 As $\phi$ is contained in a compact group, all the limit points of $\{\phi^m\}$ keep $V_\alpha$ and $V_\beta$ invariant. 
 In particular, $\phi^m(L_w)\subset V_\beta$ for all $m$. Moreover, $\phi^m(S^m(L_1))\to V_\alpha$ as $m\to\infty$. 
 Thus all the limit points of $\{T^m(W\cap V_{\alpha\beta})\}$ contain $V_\alpha$ as $m\to\infty$
 
 The limit points of $\phi^{-m}(S^{-m}(L_1))$ are $\psi(L'_w)$ as $m\to\infty$, where $\psi$ is any limit point of 
 $\{\phi^m\mid -\infty< m \leq -1\}$. Therefore the limit points of  $\{T^{-m}(W\cap V_{\alpha\beta})\}$ are contained in 
 $V_\beta$ as $m\to\infty$. As $\dim(V_\beta)=2=\dim(W\cap V_{\alpha\beta})$, we get that 
 $T^{-m}(W\cap V_{\alpha\beta})\to V_\beta$ as $m\to\infty$.
 
Now we have that all the limit points of $\{T^m(W)\}$ contain $V_\alpha$ as $m\to\infty$, and contain $V_\beta$ as $m\to -\infty$, i.e.\ for 
any rational subspace $W\in R_{n-1}$, if $T^{m_k}(W)\to H$ then $V_\alpha\subset H$ when $m_k\to \infty$, and $V_\beta\subset H$ when 
$m_k\to -\infty$. Consider the set $\mathcal{A}=\{H\in \mathbb{H}_{n-1}\mid V_{\alpha}\subset H\ \mbox{ or }\ V_\beta\subset H\}$. Then 
$\mathcal{A}$ is a proper closed (compact) subspace of $\mathbb{H}_{n-1}$. As $R_{n-1}$ is dense in $\mathbb{H}_{n-1}$, using 
similar arguments as in Case I, we can find infinitely many rational subspaces in $R_{n-1}$ whose $T$-orbits are disjoint. 

\medskip
 \noindent{\bf Case III}: Now suppose both the eigenvalues $\alpha$ and $\beta$ of $T$ (as above) are complex with  
 $|\alpha|>1$ and $|\beta|<1$. Let $V_{\alpha}$ (resp.\ $V_{\beta}$) be a $T$-invariant $2$-dimensional subspace such that 
 $T|_{V_\alpha}$ (resp.\ $T|_{V_\beta}$) have $\alpha$ and $\bar\alpha$ (resp.\ $\beta$ and $\bar\beta$) as eigenvalues. Let 
 $V_{\alpha \beta}=V_{\alpha}+V_{\beta}$. Then $\dim(V_{\alpha\beta})=4$ and $T$ keeps $V_{\alpha\beta}$ invariant. 
 Now we have that $T|_{V_{\alpha\beta}}=\phi S=S\phi$ where, $S|_{V_\alpha}=|\alpha|\Id$, $S|_{V_\beta}=|\beta|\Id$, and 
 each of the maps $\phi|_{V_\alpha}$ and $\phi|_{V_\beta}$ generate a relatively compact group in $\GL(V_\alpha)$ and 
 $\GL(V_\beta)$ respectively. In particular $\phi$ is contained in a compact subgroup of $\GL(V_{\alpha\beta})$.

For any rational subspace $W\in R_{n-1}$, $V_\alpha\not\subset W$ and $V_\beta\not\subset W$. 
Let $L_{\alpha_1}=W \cap V_{\alpha}$ and let $L_{\beta_1}=W\cap V_{\beta}$. Then 
$V_{\alpha}=L_{\alpha_1}+L_{\alpha_2}$ and $V_{\beta}=L_{\beta_1}+L_{\beta_2}$ for some one-dimensional subspaces 
$L_{\alpha_2}$ in $V_\alpha$ and $L_{\beta_2}$ in $V_\beta$. Let $V_2=L_{\alpha_2}+ L_{\beta_2}$. Then $\dim(V_2)=2$ 
and $V_2$ is $S$-invariant. Let $S_2:=S|_{V_2}$. Then $S_2\in \GL(V_2)$ and $S_2$ and $S_2^{-1}$ are both proximal. 
Let $L_2=W\cap V_2$. Then $\dim(L_2)=1$ as neither $L_{\alpha_2}$ nor $L_{\beta_2}$ is contained in $W$. As $S_2$ 
is proximal, we get that $S^m(L_2)\to L_{\alpha_2}$ and $S^{-m}(L_2)\to L_{\beta_2}$ in $\Sub^p_{\R^n}$ as $m\to\infty$. 
Now 
$$
S^m(W\cap V_{\alpha\beta})=S^m(L_{\alpha_1}+ L_{\beta_1}+L_2)=L_{\alpha_1}+L_{\beta_1}+S^m(L_2). 
$$
Therefore, as $m\to\infty$,
\begin{eqnarray*} S^m(W\cap V_{\alpha\beta})\to L_{\alpha_1}+L_{\beta_1}+L_{\alpha_2} & = & V_\alpha+L_{\beta_1},\mbox{ and }\\
S^{-m}(W\cap V_{\alpha\beta})\to L_{\alpha_1}+L_{\beta_1}+L_{\beta_2} & = &L_{\alpha_1}+V_\beta.
\end{eqnarray*}
As $\phi$ keeps $V_\alpha$ and $V_\beta$ invariant, so does every limit point of $\{\phi^m\}_{m\in\Z}$. As $\{\phi^m\}_{m\in\Z}$ 
is relatively compact and $T^m(W\cap V_{\alpha\beta})=\phi^mS^m(W\cap V_{\alpha\beta})$ for all $m\in\Z$, we have that 
limit points of $\{T^m(W\cap V_{\alpha\beta})\mid m\in\N\}$ contain $V_\alpha$ and limit points of 
$\{T^{-m}(W\cap V_{\alpha\beta})\mid m\in\N\}$ contain $V_\beta$. Thus limit points of $\{T^m(W)\mid m\in\N\}$ contain 
$V_\alpha$ and $\{T^{-m}(W)\mid m\in\N\}$ contain $V_\beta$ for every rational subspace $W\in R_{n-1}$. 

Let $\mathcal{A}=\{H\in \mathbb{H}_{n-1} \mid V_{\alpha}\subset H \ \mbox{ or } \ V_{\beta}\subset H\}$. Then $\mathcal{A}$ is a proper 
closed (compact) subset of $\mathbb{H}_{n-1}$. As $R_{n-1}$ is dense in $\mathbb{H}_{n-1}$, using similar arguments as in Case I, 
we can find infinitely many rational subspaces in $R_{n-1}$ whose $T$-orbits are disjoint. 
 \end{proof}
 
Now we prove \Cref{expansive-t} where we will use \Cref{rational} for a particular class of automorphisms.

\begin{proof}[Proof of \Cref{expansive-t}] Let $G=\mT^n$, the $n$-torus, for any $n\geq 2$ and let $T\in\Aut(G)$. We want 
to show that the $T$-action on $\spg$ is not expansive. If $T^m=\Id$ for some $m\in\N$, then by \Cref{expa-prop}\,(4), 
the $T$-action on $\spg$ is not expansive as $\spg$ is infinite. Now suppose that $T^m\ne\Id$ for any $m\in\N$. 

Recall that $\mfS_{n-1}$ is the collection of all $(n-1)$-dimensional sub-tori of $G=\mT^n$, and 
$\mfH_{n-1}= \mathfrak{S}_{n-1}\cup \{G\}$, both of these sets are $T$-invariant subsets of $\spg$, and $\mfH_{n-1}$ is closed 
in $\spg$, and hence it is compact. Moreover, every $H\in \mathfrak{S}_{n-1}$ is isolated in $\mfH_{n-1}$ (cf.\  \Cref{isol}); 
in particular, $\{H\}$ is open in $\mfH_{n-1}$. Note that for every $H$ in $\mfS_{n-1}$, $T^m(H)\to G$, if $m\to \pm\infty$ 
unless $T^m(H)=H$ for some $m\in\N$. However, since $\spg$ is countable (see \Cref{subptn}\,(ii)), we are not able to use Theorem 1 of 
\cite{Re}, even if $T^m(H)\ne H$, $m\in\N$, for infinitely many $H\in\mfS_{n-1}$. 

\medskip
\noindent{\bf Step 1:} Let ${\bf d}$ be the metric on $\spg$. If possible, suppose that the $T$-action on $\spg$ is expansive. Then the 
$T$-action on $\mfH_{n-1}$ is also expansive; suppose $\epsilon>0$ is an expansive constant for this action. For any $H\in \mfS_{n-1} $, 
 there exists $m\in \Z$ such that ${\bf d}(T^m(H),G)>\epsilon$. Let $B(G,\epsilon)$ be the ball of radius $\epsilon$ centered at $G$ in 
 $\spg$. Consider the collection $\{\{H\}\mid H\in\mfS_{n-1}\}\cup \{B(G, \epsilon)$\};  
it is an open cover of $\mfH_{n-1}$. As $\mfH_{n-1}$ is compact, there exist $H_1, \dots,H_k$ in $\mfS_{n-1}$ such that 
 $\mfH_{n-1}=B(G,\epsilon)\cup \{H_1\}\cup \dots \cup \{H_k\}$. As the $T$-action on $\spg$ is expansive with an expansive 
 constant $\epsilon$, for any $H\in \mathfrak{S}_{n-1}$, there exist $m\in \Z$ (which depends on $H$) such that $T^m(H)=H_i$ 
 for some $i \in \{1,\dots,k\}$. This implies that $T$ has finitely many orbits in $\mfS_{n-1}$. We will now show that $T$ has infinitely 
 many disjoint orbits in $\mfS_{n-1}$, which would contradict the expansivity of $T$.  

\medskip
\noindent{\bf Step 2:} Suppose that all the eigenvalues of $T$ have absolute value $1$, i.e.\ $T$ is distal. 
As $T\in\GL(n,\Z)$, some power of $T$ is unipotent (see e.g.\ Lemma 2.5 of \cite{Ab2}). The statement that there are infinitely many 
$(n-1)$-dimensional sub-tori with disjoint $T$-orbits is equivalent to the statement that there are infinitely many $(n-1)$-dimensional 
sub-tori with disjoint $T^m$-orbits for any $m\in\N$. Without loss of any generality, we may assume that $T$ is unipotent and that 
$T\neq \Id$. By Proposition 3.10 of \cite{ShY3}, there exist non-trivial  closed connected $T$-invariant subgroups 
$\{e\}=K_0\subsetneq K_1\cdots\subsetneq K_{l+1}=G$ of $\mT^n$ such that $T$ acts trivially on $K_i/K_{i-1}$, $1\leq i\leq l+1$ 
and $T$ does not act trivially on $K_i/K_{i-2}$, $2\leq i\leq l+1$ (see also Lemma 2.5 of \cite{Ab2}). 
Note that $l\ne 0$ as $T\neq\Id$. If $\dim(K_l)\leq n-2$, then $\dim(G/K_l)\geq 2$, and we can choose infinitely many distinct tori 
(closed connected subgroups) $B_m$ of co-dimension $1$ in $G/K_l$, $m\in\N$. Let $H_m$ be a sub-torus in $G$ containing 
$K_l$ such that $H_m/K_l=B_m$ for each $m$. Then $\dim(H_m)=n-1$ and $H_m$ is $T$-invariant for every $m$. 
Moreover, $H_m$'s are distinct as $B_m$'s are so and $\{H_m\}$ is a $T$-orbit in $\mfS_{n-1}$ for each $m$. 

 Now suppose $\dim(K_l)=n-1$. There exists a closed one-parameter subgroup $C_0=\{x_t\}$ such that $G=C_0\times K_l$. Now
$T(x_t)=x_ty_t$, $t\in\R$, for some non-trivial closed one-parameter subgroup $\{y_t\}\subset K_l$. For $m\in \N$, 
let $C_m :=\{x_t y_{t/m}\}$.

Suppose $\dim(K_l)=1$. Then $l=1$ and $T|_{K_l}=\Id$.  Then $T(\{x_t\})=\{x_t y_t\}$ and $T^k(\{x_t\})=\{x_t y_t^k\}=\{x_t y_{kt}\}$ 
and the $T$-orbit of $C_0$ is $\{\{x_t y_{kt}\}\mid k\in \Z\}$. Moreover, $T^k(C_m)=\{x_t y_{(k+1/m)t}\}$ and 
the $T$-orbit of $C_m$ is $\{\{x_t y_{(k+1/m)t}\} \mid k\in \Z\}$. Observe that the $T$-orbits of $C_m$'s are disjoint. 

Now suppose $\dim(K_l)\geq 2$. We can choose a $T$-invariant sub-torus $H$ such that $K_{l-1}\subset H\subset K_l$ and 
$K_l=\{y_t\}\times H$. Here, $H=K_{l-1}$ if $\dim(K_l/K_{l-1})=1$. Note that $T(y_t)\in y_tK_{l-1}\subset y_tH$, $t\in\R$. 
Let $H_m:=C_mH$, $m\in\N$. Then $\dim(H_m)=n-1$ and $H_m\in\mfS_{n-1}$, $m\in\N$. As $H$ is $T$-invariant, 
it is easy to see that the $T$-orbits of $H_k$ and $H_m$ are disjoint if $k\neq m$. Thus, if $T$ is distal, then $T$ 
has infinitely many disjoint orbits in $\mfS_{n-1}$.

\medskip
\noindent{\bf Step 3:} Now suppose $T$ admits an eigenvalue with absolute value other than $1$, i.e.\  $T$ is not distal. 
Let $\pi:\R^n\to \mT^n$ be the natural projection with $\ker\pi=\Z^n$. Then we have $T\in\GL(n,\Z)$ 
as a linear automorphism of $\R^n$ with $\pi\,\circ T=T\circ\,\pi$. Recall that $R_{n-1}$, the set of all $(n-1)$-dimensional 
rational subspace of $\R^n$, is a $T$-invariant subspace of $\Sub_{\R^n}$, then $\pi(R_{n-1})=\mfS_{n-1}$, and the 
map from $R_{n-1}$ to $\mfS_{n-1}$ induced by $\pi$ is bijective. 

Suppose $T$ does not keep any non-trivial proper rational subspace of $\R^n$ invariant. Then by \Cref{rational} in this case, 
there are infinitely many rational subspaces in $R_{n-1}$ with disjoint $T$-orbits. Then there are infinitely many 
$(n-1)$-dimensional tori in $\spt$ with disjoint $T$-orbits. 

Suppose $n=2$, i.e.\ $\dim(G)=2$. Then since $T$ does not have an eigenvalue of absolute value $1$, both the eigenvalues 
of $T$ are real with absolute value other than $1$. We now show that $T$ does not keep any non-trivial proper rational subspace 
of $\R^2$ invariant. If possible suppose $T(W)=W$ for a proper non-trivial rational subspace $W$ of $\R^2$, 
then $\dim(W)=1$. Hence $W\cap \Z^2$ is a $T$-invariant infinite cyclic subgroup. Thus $T(z)=\pm z$, for all 
$z\in W\cap\Z^2$ and hence $T$ has an eigenvalue which is either $1$ or $-1$. As $\det T=\pm 1$, we have that  
the other eigenvalue of $T$ is also $1$ or $-1$. This contradicts the assumption above. 

Now by \Cref{rational}, $T$ has infinitely many $1$-dimensional rational subspaces with disjoint $T$-orbits. Thus the 
assertion that $T$ has infinitely many co-dimension one sub-tori with disjoint $T$-orbits holds for all $T\in\Aut(G)$ for $n=2$. 

Suppose $n\geq 3$. Suppose the assertion that every $T\in\Aut(G)$ has infinitely many disjoint orbits in $\mfS_{k-1}$ 
holds for any torus $G$ with dimension $k$ such that $1<k<n$. Now suppose $G$ is such that $\dim(G)=n$. If $T$ is distal, then 
the assertion holds as shown in Step 2. If $T$ does not keep any non-trivial proper rational 
subspace of $\R^n$ invariant, equivalently, if $T$ does not keep any proper sub-tori invariant, then the assertion holds 
by \Cref{rational}. as shown above.

Now suppose that $T$ keeps a nonzero proper rational subspace of $\R^n$ invariant. Then 
there is a proper sub-torus $H$ on $G$ such that $T(H)=H$. Then $\dim(H)<n$. Let $\T\in\Aut(G/H)$ be the automorphism induced by $T$. 
Suppose the dimension of $G/H$ is greater than or equal to $2$. Since $\dim(G/H)<n$, by the induction hypothesis, $G/H$ has infinitely 
many sub-tori (say) $B_m$ of co-dimension $1$ with disjoint $\T$-orbits. Let $H_m$ be the sub-tori of $G$ containing $H$ such that 
$H_m/H=B_m$. Then $\dim(H_m)=n-1$ and $H_m$'s have disjoint $T$-orbits. 

Now suppose $\dim(G/H)=1$. Then the corresponding automorphism $\T$ of $G/H$, and hence, $T$ has a real eigenvalue 
which is $\pm 1$. We can choose a one-dimensional sub-torus, say, $M$ of $G$ which is $T$-invariant. Then we have that 
$\dim(G/M)=n-1\geq 2$. Thus arguing as above for $M$ instead of $H$, we have that $G$ has infinitely many 
sub-tori of dimension $n-1$ with disjoint $T$-orbits. By induction, the assertion that there are infinitely many sub-tori of 
co-dimension $1$ in $G$ with disjoint $T$-orbits holds. 

As noted at the end of Step 1, it follows that the $T$-action on $\spg$ is not expansive. 
\end{proof}

The proof of  \Cref{expansive-t} actually proves a stronger statement that the $T$-action on $\mfH_{n-1}$ is not expansive, where 
$\mfH_{n-1}$ consists of the whole group $G$ and all the $(n-1)$-dimensional sub-tori of $G=\mT^n$, $n\geq 2$. 

\begin{remark} \label{rem1} As shown in the proof of \Cref{expansive-t}, for any $T\in\GL(n,\Z)$, $n\geq 2$, there are infinitely many 
sub-tori of co-dimension $1$ in $\mT^n$ with disjoint $T$-orbits. This is equivalent to the statement that $T$ has infinitely many disjoint 
$T$-orbits in $R_{n-1}$, the space of $(n-1)$-dimensional rational subspaces of $\R^n$. Thus \Cref{rational} holds for any 
$T\in\GL(n,\Z)$ without the condition on $T$ mentioned there. 
\end{remark}

The following corollary follows easily from \Cref{expansive-t}. We give a proof for the sake of completeness. 

\begin{corollary} \label{Lie-t}
Let $G$ be a connected Lie group such that it contains a central torus of dimension at least $2$. Then $G$ does not admit any automorphism 
that acts expansively on $\spg$.
\end{corollary}

\begin{proof} Let $C$ be the maximal central torus in $G$. By the hypothesis, the dimension of $C$ is at least $2$. Let $T\in\Aut(G)$. 
Then $T(C)=C$ and by \Cref{expansive-t}, $T$ does not act expansively on $\Sub^p_C$. As $\Sub^p_C$ is a closed subspace of $\spg$, 
by \Cref{expa-prop}\,(5), we get that $T$ does not act expansively on $\spg$. 
\end{proof}

\subsection*{Acknowledgments.} We thank Amala Bhave for valuable discussions. \\Riddhi Shah acknowledges the ARG-MATRICS grant from ANRF, India\\ 
(ANRF/ARGM/2025/000695/MTR). 
Debamita Chatterjee acknowledges  the UGC Non-NET research fellowship, and Himanshu Lekharu acknowledges 
the CSIR-JRF research fellowship from CSIR, Govt.\ of India. 

{\advance}

\end{document}